\documentclass[11pt]{article}

\usepackage[margin=1in]{geometry}
\usepackage{amsmath,amssymb,amsthm,mathtools}
\usepackage{booktabs}
\usepackage{enumitem}
\usepackage{hyperref}
\usepackage{bm}
\usepackage{xcolor}

\hypersetup{colorlinks=true,linkcolor=blue,citecolor=blue,urlcolor=blue}
\allowdisplaybreaks
\numberwithin{equation}{section}

\newtheorem{theorem}{Theorem}[section]

\newtheorem{remark}[theorem]{Remark}
\newtheorem{assumption}[theorem]{Assumption}

\newcommand{\R}{\mathbb{R}}
\newcommand{\ip}[2]{\left\langle #1,#2\right\rangle}
\newcommand{\norm}[1]{\left\lVert #1\right\rVert}
\newcommand{\normM}[2]{\left\lVert #1\right\rVert_{#2}}
\newcommand{\prox}{\operatorname{prox}}
\newcommand{\calL}{\mathcal{L}}
\newcommand{\calG}{\mathcal{G}}
\newcommand{\calX}{\mathcal{X}}
\newcommand{\calS}{\mathcal{S}}
\newcommand{\Id}{I}
\newcommand{\PD}{\mathrm{PD3O}}
\newcommand{\PDDY}{\mathrm{PDDY}}
\newcommand{\PAPC}{\mathrm{PAPC}}
\newcommand{\CV}{\mathrm{CV}}

\title{Unified Ergodic Primal-Dual Gap Rates\\with Unhalved Primal Stepsizes}
\author{Sirong Dai and Ming Yan\footnote{School of Data Science, The Chinese University of Hong Kong, Shenzhen}}
\date{\today}

\begin{document}
\maketitle

\begin{abstract}
We study ergodic primal-dual gap rates for first-order primal-dual methods applied to
\[
    \min_x f(x)+g(x)+h(Ax),
\]
where $f$ is smooth and convex, $g$ and $h$ are proper, closed, convex functions, and $A$ is linear.  Standard gap-rate proofs often impose the halved smooth-stepsize condition $\tau \le 1/L$, even though the corresponding convergence theory allows the larger range $\tau <2/L$.  We introduce a residual-to-gap transfer principle: positive residual terms in the one-step gap inequality are controlled by the decrease of a Lyapunov function.  This yields $O(1/K)$ ergodic primal-dual gap bounds with the unhalved primal stepsize $\tau <2/L$ for Condat--V\~u, PD3O, AFBA/PDDY, and PAPC/PDFP$^2$O, under their algorithm-dependent product conditions. We also give a two-dimensional counterexample showing that the fully separated rectangle $\tau <2/L$, $\tau\eta\|A\|^2<4/3$ cannot hold in the general three-function setting.
\end{abstract}

\section{Introduction}
\label{sec:introduction}

We consider the composite convex optimization problem
\begin{equation}
    \min_{x\in\mathbb R^n} F(x):=f(x)+g(x)+h(Ax),
    \label{eq:intro_problem}
\end{equation}
where $f$ is convex and differentiable with an $L$-Lipschitz continuous gradient, $g$ and $h$ are proper, closed, convex functions, and $A:\mathbb R^n\to\mathbb R^m$ is linear. Its saddle-point formulation is
\begin{equation}
    \min_{x\in\mathbb R^n}\max_{s\in\mathbb R^m}
    \mathcal L(x,s):= f(x)+g(x)+\langle Ax,s\rangle-h^*(s).
    \label{eq:intro_saddle}
\end{equation}
Problems of the form \eqref{eq:intro_problem} are standard in imaging and signal processing, machine learning, and distributed optimization. They have motivated a large family of first-order primal-dual splitting schemes, including Chambolle--Pock~\cite{ChambollePock2011}, Condat--V\~u~\cite{ChambollePock2016,Condat2013,Vu2013}, primal-dual fixed-point~\cite{ChenHuangZhang2016}, PAPC/PDFP$^2$O~\cite{drori2015simple,chen2013primal,loris2011generalization}, PD3O~\cite{Yan2018PD3O},
and AFBA/PDDY~\cite{LatafatPatrinos2017,salim2022dualize}. These algorithms are attractive because they separate the smooth gradient step, the nonsmooth proximal steps, and the applications of the linear map $A$ and its adjoint.

A central issue in the analysis of these methods is the admissible stepsize range. Let $\tau$ and $\eta$ denote the primal and dual stepsizes, and set
\begin{equation*}
    \alpha:=\frac{\tau L}{2},
    \qquad
    \beta:=\tau\eta\|A\|^2.
\end{equation*}
The bound $\alpha<1$, equivalently $\tau<2/L$, is unavoidable in the gradient-descent limit. The product condition on $\beta$, however, depends delicately on the algorithm and on the structure of the objective. For Chambolle--Pock and PAPC/PDFP$^2$O, the product bound can be enlarged from the classical $\beta\le1$ to the sharp condition $\beta<4/3$ \cite{LiYan2021}. For AFBA/PDDY, a recent analysis gives the coupled condition
\begin{equation*}
\alpha<1,\quad \beta\le\dfrac{4-2\alpha}{3-\alpha},
\end{equation*}
which interpolates between the classical product bound and the large-step regime~\cite{YanLi2024}. In contrast, the sufficient condition for PD3O remains $\alpha<1$ and $\beta\le1$.

For the convergence of the iterates, the difference between these regions is now relatively well understood. The situation for primal-dual gap rates is less satisfactory. Many standard ergodic gap proofs proceed from a one-step telescoping inequality and then force all residual terms to be nonpositive. When a smooth term is present, this often requires
\begin{equation}
    \frac{L}{2}-\frac{1}{2\tau}\le0,
    \qquad\text{that is,}\qquad
    \alpha \le \frac12.
    \label{eq:intro_halved_stepsize}
\end{equation}
Thus, the rate proof uses only half of the primal stepsize range allowed by the convergence proof. This phenomenon already appears in several primal-dual gap analyses: a clean $O(1/K)$ estimate is obtained, but at the cost of the halved smooth stepsize. The purpose of this paper is to remove this artificial loss while preserving the correct algorithm-dependent product restrictions.

The main idea is simple. We do not require the residual in the one-step gap inequality to be nonpositive. Instead, we show that the positive part of this residual is summable because it is controlled by the descent of a Lyapunov function. In abstract form, suppose
\begin{equation}
    \mathcal L(u^k,s)-\mathcal L(x,t^k)
    \le D_k(x,s)-D_{k+1}(x,s)+R_k,
    \label{eq:intro_gap_decomposition}
\end{equation}
where $(u^k,t^k)$ are the primal and dual outputs used for averaging. If, for a Lyapunov $V_k$, one has
\begin{equation}
    R_k\le \rho\,(V_k-V_{k+1}),
    \label{eq:intro_residual_domination}
\end{equation}
then both terms telescope, and an $O(1/K)$ ergodic primal-dual gap bound follows. This residual-to-gap transfer is the organizing principle of this paper. It explains why the halved condition \eqref{eq:intro_halved_stepsize} is not intrinsic: the residual may be positive, but it can be paid for by Lyapunov descent.

Our contributions are as follows.
\begin{enumerate}
    \item We formulate an abstract residual-to-gap theorem. The theorem converts a one-step gap decomposition of the form \eqref{eq:intro_gap_decomposition}, together with the residual domination estimate \eqref{eq:intro_residual_domination}, into an ergodic $O(1/K)$ primal-dual gap bound. The statement is written for arbitrary test pairs $(x,s)$, so it yields restricted primal-dual gap estimates, not only saddle-pair estimates.

    \item We apply the abstract theorem to four algorithms: Condat--V\~u, PD3O, AFBA/PDDY, and PAPC/PDFP$^2$O. In each case, the same proof template is used: derive the one-step gap inequality, identify the residual that prevents the classical proof from using $\tau <2/L$, and dominate this residual by a corresponding Lyapunov descent.

    \item We also include a low-dimensional counterexample showing that the fully separated rectangle
    \begin{equation*}
        \alpha<1,
        \qquad
        \beta<\frac43
    \end{equation*}
    cannot hold for the general three-function AFBA/PDDY or PD3O when both $f$ and $g$ are nonzero. This justifies the coupled AFBA/PDDY region and the more conservative PD3O product condition. The $4/3$ product bound is a special feature of PAPC/PDFP$^2$O-type reductions, not a universal three-function phenomenon.
\end{enumerate}

The rest of this paper is organized as follows. Section~\ref{sec:preliminaries} introduces the problem setup and notation. Section~\ref{sec:abstract_transfer} proves the abstract residual-to-gap theorem. Section~\ref{sec:algorithms} introduces the four algorithms and states the unified gap-rate theorem. Section~\ref{sec:verification} verifies the abstract theorem for Condat--V\~u, PD3O, AFBA/PDDY, and PAPC/PDFP$^2$O. Section~\ref{sec:counterexample} gives the counterexample excluding the separated $(2,4/3)$ rectangle in the general three-function case. Section~\ref{sec:conclusion} concludes with remarks on possible extensions, including whether a lifted Lyapunov function can further relax the PD3O product condition.

\section{Problem setup and notation}\label{sec:preliminaries}
Throughout this paper, $\|A\|$ denotes the spectral norm of $A$, and
\[
    \alpha:=\frac{\tau L}{2},
    \qquad
    \beta:=\tau\eta\|A\|^2.
\]
For a self-adjoint matrix $M$, define
\[
    \normM{z}{M}^2:=\ip{z}{Mz}.
\]
If $M$ is indefinite, this is only a quadratic form. For any real number $a$, we denote $(a)_+\coloneqq \max\{a,0\}$.

\begin{assumption}\label{ass:standing}
The function $f:\R^n\to\R$ is convex and differentiable with $L$-Lipschitz continuous gradient.  The functions $g:\R^n\to(-\infty,+\infty]$ and $h:\R^m\to(-\infty,+\infty]$ are proper, closed, and convex.  The saddle problem \eqref{eq:intro_saddle} has at least one saddle point $(x^\star,s^\star)$.
\end{assumption}

The saddle point condition is
\begin{equation*}
    0\in \nabla f(x^\star)+\partial g(x^\star)+A^\top s^\star,
    \qquad
    0\in \partial h^*(s^\star)-Ax^\star.
\end{equation*}
We use the following two standard inequalities.  For $L$-smooth convex $f$,
\begin{equation*}
    f(u)\le f(v)+\ip{\nabla f(v)}{u-v}+\frac L2\norm{u-v}^2,
\end{equation*}
whereas cocoercivity gives
\begin{equation*}
    \ip{u-v}{\nabla f(u)-\nabla f(v)}
    \ge \frac1L\norm{\nabla f(u)-\nabla f(v)}^2.
\end{equation*}
Equivalently,
\begin{equation}\label{eq:smooth-bregman}
    f(u)-f(v)
    \le \ip{\nabla f(u)}{u-v}
       -\frac1{2L}\norm{\nabla f(u)-\nabla f(v)}^2.
\end{equation}

\section{Abstract residual-to-gap theorem}\label{sec:abstract_transfer}

The following theorem is used for all four algorithms.  It separates the one-step gap calculation from the algorithm-specific convergence descent.

\begin{theorem}[Residual-to-gap transfer]\label{thm:abstract-transfer}
Let $\calL:X\times S\to(-\infty,+\infty]$ be convex in its first variable and concave in its second variable, and suppose that $(x^\star,s^\star)$ is a saddle point.  Let $\{u^k\}_{k\ge0}\subset X$ and $\{t^k\}_{k\ge0}\subset S$ be the primal and dual outputs used in the gap estimate.

Assume that, for every fixed test pair $(x,s)$, there exist a distance-like sequence $D_k(x,s)$, a residual sequence $R_k$, a nonnegative nonincreasing Lyapunov sequence $V_k$, and constants $\rho\ge0$, $\kappa\ge0$, and $B(x,s)\ge0$ such that, for all $k\ge0$,
\begin{align}
    \calL(u^k,s)-\calL(x,t^k)
    &\le D_k(x,s)-D_{k+1}(x,s)+R_k, \label{eq:H1}\\
    R_k&\le \rho\bigl(V_k-V_{k+1}\bigr), \label{eq:H2}\\
    D_k(x,s)&\ge -\kappa V_k-B(x,s). \label{eq:H3}
\end{align}
Define
\begin{equation*}
    \bar u^K:=\frac1K\sum_{k=0}^{K-1}u^k,
    \qquad
    \bar t^K:=\frac1K\sum_{k=0}^{K-1}t^k.
\end{equation*}
Then, for all $K\ge1$,
\begin{equation}\label{eq:abstract-bound}
    \calL(\bar u^K,s)-\calL(x,\bar t^K)
    \le
    \frac{D_0(x,s)+B(x,s)+(\rho+\kappa)V_0}{K}.
\end{equation}
Consequently, for any test sets $\calX\subset X$ and $\calS\subset S$ for which the right-hand side is uniformly finite,
\begin{equation*}
    \calG_{\calX,\calS}(\bar u^K,\bar t^K)
    :=\sup_{x\in\calX,\,s\in\calS}\{\calL(\bar u^K,s)-\calL(x,\bar t^K)\}
    =O(1/K).
\end{equation*}
If $(x^\star,s^\star)\in\calX\times\calS$, then $\calG_{\calX,\calS}(\bar u^K,\bar t^K)\ge0$.
\end{theorem}

\begin{proof}
Let
\[
    G_k(x,s):=\calL(u^k,s)-\calL(x,t^k).
\]
By \eqref{eq:H1} and \eqref{eq:H2},
\begin{equation*}
    G_k(x,s)
    \le D_k(x,s)-D_{k+1}(x,s)+\rho(V_k-V_{k+1}).
\end{equation*}
Summing from $k=0$ to $K-1$ gives
\begin{equation*}
    \sum_{k=0}^{K-1}G_k(x,s)
    \le D_0(x,s)-D_K(x,s)+\rho(V_0-V_K).
\end{equation*}
Using \eqref{eq:H3} and $0\le V_K\le V_0$,
\[
    -D_K(x,s)\le \kappa V_K+B(x,s)\le \kappa V_0+B(x,s),
    \qquad
    V_0-V_K\le V_0.
\]
Therefore
\begin{equation}\label{eq:abstract-summed-final}
    \sum_{k=0}^{K-1}G_k(x,s)
    \le D_0(x,s)+B(x,s)+(\rho+\kappa)V_0.
\end{equation}
By convexity in the primal variable and concavity in the dual variable,
\[
    \calL(\bar u^K,s)\le \frac1K\sum_{k=0}^{K-1}\calL(u^k,s),
    \qquad
    \calL(x,\bar t^K)\ge \frac1K\sum_{k=0}^{K-1}\calL(x,t^k).
\]
Thus
\[
    \calL(\bar u^K,s)-\calL(x,\bar t^K)
    \le \frac1K\sum_{k=0}^{K-1}G_k(x,s),
\]
and \eqref{eq:abstract-bound} follows from \eqref{eq:abstract-summed-final}.  If the test set contains a saddle point, then
\[
    \calL(\bar u^K,s^\star)\ge \calL(x^\star,s^\star)\ge \calL(x^\star,\bar t^K),
\]
which implies nonnegativity of the restricted supremum.
\end{proof}

\begin{remark}\label{rem:why-helps}
The standard primal-dual gap proof corresponds to $R_k\le0$.  Then one may take $\rho=0$, but the proof often requires $\tau \le1/L$.  Theorem~\ref{thm:abstract-transfer} allows $R_k>0$ as long as the residual can be paid for by the descent $V_k-V_{k+1}$.  This is the mechanism that removes the halved smooth stepsize.
\end{remark}

\section{Algorithms and unified theorem}\label{sec:algorithms}

We now instantiate the abstract framework for four primal-dual splitting algorithms.  For each method, we specify the iteration and the primal-dual outputs used in the ergodic gap estimate.

\paragraph{Condat--V\~u.}
The Condat--V\~u iteration, also known as Chambolle--Pock with a smooth term, is
\begin{align}
    x^{k+1}
    &=\prox_{\tau g}\left(x^k-\tau\nabla f(x^k)-\tau A^\top s^k\right),\label{eq:cv-x}\\
    s^{k+1}
    &=\prox_{\eta h^*}\left(s^k+\eta A(2x^{k+1}-x^k)\right).\label{eq:cv-s}
\end{align}
The primal and dual outputs used in the analysis are
\[
    u_{\CV}^k:=x^{k+1},
    \qquad
    t_{\CV}^k:=s^{k+1}.
\]

\paragraph{PD3O.}
The PD3O iteration is
\begin{align}
    x^k&=\prox_{\tau g}(z^k), \label{eq:pd3o-x}\\
    s^{k+1}
    &=\prox_{\eta h^*}\Bigl((\Id-\tau\eta AA^\top)s^k
      +\eta A(2x^k-z^k-\tau\nabla f(x^k))\Bigr), \label{eq:pd3o-s}\\
    z^{k+1}
    &=x^k-\tau\nabla f(x^k)-\tau A^\top s^{k+1}. \label{eq:pd3o-z}
\end{align}
The gap outputs are
\[
    u_{\PD}^k:=x^k,
    \qquad
    t_{\PD}^k:=s^{k+1}.
\]

\paragraph{PDDY/AFBA.}
The PDDY/AFBA iteration is
\begin{align}
    s^{k+1}&=\prox_{\eta h^*}(s^k+\eta Ay^k), \label{eq:pddy-s}\\
    x^{k+1}&=y^k-\tau A^\top(s^{k+1}-s^k), \label{eq:pddy-x}\\
    y^{k+1}&=\prox_{\tau g}\bigl(x^{k+1}-\tau A^\top s^{k+1}-\tau\nabla f(x^{k+1})\bigr). \label{eq:pddy-y}
\end{align}
The gap outputs are
\[
    u_{\PDDY}^k:=y^k,
    \qquad
    t_{\PDDY}^k:=s^{k+1}.
\]

\paragraph{PAPC/PDFP$^2$O.}
When $g\equiv0$, PD3O and PDDY reduce to PAPC/PDFP$^2$O.  The iteration is
\begin{align}
    s^{k+1}
    &=\prox_{\eta h^*}\Bigl((\Id-\tau\eta AA^\top)s^k
       +\eta A(x^k-\tau\nabla f(x^k))\Bigr), \label{eq:papc-s}\\
    x^{k+1}
    &=x^k-\tau\nabla f(x^k)-\tau A^\top s^{k+1}. \label{eq:papc-x}
\end{align}
The gap outputs are
\[
    u_{\PAPC}^k:=x^{k+1},
    \qquad
    t_{\PAPC}^k:=s^{k+1}.
\]

The role of the next theorem is to show that each algorithm admits a one-step gap decomposition whose residual is controlled by a Lyapunov function.  The algorithm-specific work is therefore confined to verifying the three hypotheses of Theorem~\ref{thm:abstract-transfer}.

\begin{theorem}[Unified primal-dual gap rates]\label{thm:unified}
Let Assumption~\ref{ass:standing} hold.  For each algorithm, define the ergodic averages
\begin{equation*}
    \bar u_{\mathcal A}^K:=\frac1K\sum_{k=0}^{K-1}u_{\mathcal A}^k,
    \qquad
    \bar s^K:=\frac1K\sum_{k=0}^{K-1}s^{k+1}.
\end{equation*}
Assume the corresponding stepsize condition in Table~\ref{tab:conditions}.  Then, for every fixed test pair $(x,s)$, there exists a finite constant $C_{\mathcal A}(x,s)$ such that
\begin{equation*}
    \calL(\bar u_{\mathcal A}^K,s)-\calL(x,\bar s^K)
    \le \frac{C_{\mathcal A}(x,s)}{K},
    \qquad K\ge1.
\end{equation*}
Consequently, every restricted primal-dual gap over a test set on which $C_{\mathcal A}(x,s)$ is uniformly bounded is $O(1/K)$.  In particular,
\begin{equation*}
    0\le
    \calL(\bar u_{\mathcal A}^K,s^\star)-\calL(x^\star,\bar s^K)
    \le \frac{C_{\mathcal A}(x^\star,s^\star)}{K}.
\end{equation*}
\end{theorem}

\begin{table}[h]
\centering
\caption{Admissible regions for Theorem~\ref{thm:unified}.  Here $\alpha=\tau L/2$ and $\beta=\tau\eta\|A\|^2$.}
\label{tab:conditions}
\begin{tabular}{l|l|l}
\toprule
Algorithm & New condition & Previous condition \\
\hline
Condat--V\~u
& $\alpha+\beta< 1$
& $2\alpha+\beta\leq 1$~\cite{ChambollePock2016}\\
PD3O
& $\alpha<1,\;\beta\le1$
& $2\alpha\le 1,\;\beta\le1$~\cite{Yan2018PD3O} \\
PDDY/AFBA
& $\alpha<1,\;\beta<\dfrac{4\alpha-3+\sqrt{9-8\alpha}}{2\alpha}$
& $2\alpha\le 1,\;\beta<\dfrac{8\alpha-3+\sqrt{9-16\alpha}}{4\alpha}$~\cite{YanLi2024}\\
PAPC/PDFP$^2$O ($g\equiv0$)
& $\alpha<1,\;\beta<4/3$
& $2\alpha\le1,\;\beta\le 1 $~\cite{drori2015simple}\\
\hline
\end{tabular}

\end{table}

\section{Algorithm-specific verification}\label{sec:verification}

This section proves Theorem~\ref{thm:unified}.  Each subsection verifies the three assumptions of Theorem~\ref{thm:abstract-transfer}: a one-step gap inequality, a residual domination estimate, and a lower bound on the distance term.

Let
\[
    \delta_x^k:=x^k-x^{k+1},\qquad 
    \delta_z^k:=z^k-z^{k+1},\qquad
    \delta_s^k:=s^k-s^{k+1}.
\]

\subsection{Verification for Condat--V\~u}\label{sec:condat-vu}
Define the block operators
\begin{equation*}
    \mathcal M:=
    \begin{pmatrix}
        \tau^{-1}\Id & -A^\top\\
        -A & \eta^{-1}\Id
    \end{pmatrix}.
\end{equation*}
Since \(\alpha+\beta<1\), we have \(\beta<1\), and hence \(\mathcal M\succ0.\)
For any \(u,v\),
\[
    \frac1\eta\norm{v}^2-2\ip{Au}{v} \ge -\eta\norm{Au}^2 \ge -\eta\|A\|^2\norm{u}^2,
\]
and hence
\[
    \frac12\normM{(u,v)}{\mathcal M}^2 -\frac{L}{4}\norm{u}^2 \ge \left(\frac1{2\tau}-\frac{L}{4}-\frac\eta2\|A\|^2\right)\norm{u}^2 = \frac{1-\alpha-\beta}{2\tau}\norm{u}^2.
\]

The distance term for the one-step gap is
\begin{equation*}
\begin{aligned}
    D_k^{\CV}(x,s) &:= \frac12 \normM{(x^k-x,s^k-s)}{\mathcal M}^2
\end{aligned}
\end{equation*}
Because \(\mathcal M\succ0\), we have \(D_k^{\CV}(x,s)\ge0\).
The positive Lyapunov is centered at a saddle point:
\begin{equation*}
    V_k^{\CV} := D_k^{\CV}(x^\star,s^\star) = \frac12 \normM{(x^k-x^\star,s^k-s^\star)}{\mathcal M}^2.
\end{equation*}

\subsubsection{One-step gap inequality}

The proximal steps imply
\begin{align*}
    q_g^{k+1} &:= \frac1\tau\delta_x^k-\nabla f(x^k)-A^\top s^k \in\partial g(x^{k+1}),\\
    q_h^{k+1} &:= \frac1\eta\delta_s^k+A(2x^{k+1}-x^k) \in\partial h^*(s^{k+1}).
\end{align*}
Using convexity of \(g\), convexity of \(h^*\), and smoothness of \(f\) at \(x^k\), we have
\begin{align*}
    g(x^{k+1})-g(x) &\le \ip{q_g^{k+1}}{x^{k+1}-x},\\
    f(x^{k+1})-f(x) &\le \ip{\nabla f(x^k)}{x^{k+1}-x} +\frac L2\norm{\delta_x^{k}}^2,\\
    h^*(s^{k+1})-h^*(s) &\le \ip{q_h^{k+1}}{s^{k+1}-s}.
\end{align*}
Combining these three inequalities gives~\cite[Lemma 1]{ChambollePock2016}
\begin{align}\label{eq:cv-gap-pre}
    \calL(x^{k+1},s)-\calL(x,s^{k+1}) &\le \frac L2\norm{\delta_x^k}^2 +\frac1\tau\ip{\delta_x^k}{x^{k+1}-x} +\frac1\eta\ip{\delta_s^k}{s^{k+1}-s} \\
    &\qquad -\ip{A\delta_x^k}{s^{k+1}-s} -\ip{A(x^{k+1}-x)}{\delta_s^k} \notag\\
    &\le D_k^{\CV}(x,s)-D_{k+1}^{\CV}(x,s) -\frac12\normM{(\delta_x^k,\delta_s^k)}{\mathcal M}^2 +\frac L2\norm{\delta_x^k}^2.\notag
\end{align}

We have
\begin{align*}
    -\frac12\normM{(\delta_x^k,\delta_s^k)}{\mathcal M}^2 +\frac L2\norm{\delta_x^k}^2 &= \frac L4\norm{\delta_x^k}^2-\left(\frac12\normM{(\delta_x^k,\delta_s^k)}{\mathcal M}^2 -\frac L4\norm{\delta_x^k}^2\right)\\
    &\le\left(\frac{\alpha}{1-\alpha-\beta}-1\right)\left(\frac12\normM{(\delta_x^k,\delta_s^k)}{\mathcal M}^2 -\frac L4\norm{\delta_x^k}^2\right).
\end{align*}
Denote
\[
    \boxed{\rho_{\CV}:=\frac{(2\alpha+\beta-1)_+}{1-\alpha-\beta}}.
\]
Thus~\eqref{eq:cv-gap-pre} implies
\[
    \calL(x^{k+1},s)-\calL(x,s^{k+1}) \le    D_k^{\CV}(x,s)-D_{k+1}^{\CV}(x,s)
    +R_k^{\CV},
\]
where
\[
    R_k^{\CV}:=\rho_{\CV}\left(\frac12\normM{(\delta_x^k,\delta_s^k)}{\mathcal M}^2 -\frac L4\norm{\delta_x^k}^2\right).
\]

\subsubsection{Convergence Lyapunov descent}

Let
\[
    q_g^\star:=-\nabla f(x^\star)-A^\top s^\star\in\partial g(x^\star), \qquad q_h^\star:=Ax^\star\in\partial h^*(s^\star).
\]
By monotonicity of \(\partial g\) and \(\partial h^*\),
\begin{equation}\label{eq:cv-monotonicity}
    \ip{x^{k+1}-x^\star}{q_g^{k+1}-q_g^\star} + \ip{s^{k+1}-s^\star}{q_h^{k+1}-q_h^\star} \ge0.
\end{equation}
Substituting the expressions of \(q_g^{k+1}\) and \(q_h^{k+1}\), and using the block metric \(\mathcal M\), inequality~\eqref{eq:cv-monotonicity} gives
\begin{align}\label{eq:cv-fundamental}
    0\le{}& V_k^{\CV}-V_{k+1}^{\CV} -\frac12\normM{(\delta_x^k,\delta_s^k)}{\mathcal M}^2 -\ip{x^{k+1}-x^\star}{\nabla f(x^k)-\nabla f(x^\star)}.
\end{align}
We now estimate the last term by cocoercivity.  We have
\begin{align*}
    -\ip{x^{k+1}-x^\star}{\nabla f(x^k)-\nabla f(x^\star)} &= -\ip{x^k-x^\star}{\nabla f(x^k)-\nabla f(x^\star)} +\ip{\delta_x^k}{\nabla f(x^k)-\nabla f(x^\star)}\\
    &\le -\frac1L\norm{\nabla f(x^k)-\nabla f(x^\star)}^2 +\ip{\delta_x^k}{\nabla f(x^k)-\nabla f(x^\star)}\\
    &\le \frac L4\norm{\delta_x^k}^2.
\end{align*}
Therefore~\eqref{eq:cv-fundamental} implies
\begin{align*}
    \frac12\normM{(\delta_x^k,\delta_s^k)}{\mathcal M}^2 -\frac L4\norm{\delta_x^k}^2 \le V_k^{\CV}-V_{k+1}^{\CV}.
\end{align*}
Hence
\begin{equation*}
    R_k^{\CV} \le \rho_{\CV}(V_k^{\CV}-V_{k+1}^{\CV}),
\end{equation*}

Since \(D_k^{\CV}(x,s)\ge0\), Theorem~\ref{thm:abstract-transfer} gives
\begin{equation*}
    \calL(\bar x^K,s)-\calL(x,\bar s^K) \le  \frac{D_0^{\CV}(x,s)+\rho_{\CV}V_0^{\CV}}{K},
\end{equation*}
where
\[
    \bar x^K:=\frac1K\sum_{k=0}^{K-1}x^{k+1},    \qquad     \bar s^K:=\frac1K\sum_{k=0}^{K-1}s^{k+1}.
\]
This proves the Condat--V\~u part of Theorem~\ref{thm:unified}.

\subsection{Verification for PD3O}\label{sec:pd3o}

Set
\begin{equation*}
    H:=\Id-\tau\eta AA^\top.
\end{equation*}
Under $\beta\le1$, $H\succeq0$.  For arbitrary test points $(x,s)$, define
\begin{equation*}
    z(x,s):=x-\tau\nabla f(x)-\tau A^\top s.
\end{equation*}
At the saddle point, write
\begin{equation*}
    z^\star:=z(x^\star,s^\star)=x^\star-\tau\nabla f(x^\star)-\tau A^\top s^\star.
\end{equation*}
Define
\begin{align*}
    D_k^{\PD}(x,s) :=&\frac1{2\tau}\norm{z^k-z(x,s)}^2
      +\frac1{2\eta}\normM{s^k-s}{H}^2,\\
    V_k^{\PD}:=&D_k^{\PD}(x^\star,s^\star).
\end{align*}
%Since $H\succeq0$, $D_k^{\PD}(x,s)\ge0$, so the lower-bound condition \eqref{eq:H3} holds with $\kappa=0$ and $B=0$.

\subsubsection{One-step gap inequality}

The proximal steps imply
\begin{align*}
    q_g^k&:=\frac1\tau(z^k-x^k)\in\partial g(x^k), \\
    q_h^{k+1}&:=\frac1\eta \delta_s^k-\tau AA^\top s^k
       +A(2x^k-z^k-\tau\nabla f(x^k))
       \in\partial h^*(s^{k+1}).
\end{align*}
Using the $z$-update, we get the identities
\begin{align}
    q_g^k+\nabla f(x^k)+A^\top s^{k+1}
    &=\frac1\tau\delta_z^k, \label{eq:pd3o-id1}\\
    q_h^{k+1}-Ax^k
    &=-A\delta_z^k+\frac1\eta H\delta_s^k. \label{eq:pd3o-id2}
\end{align}
Now we use the convexity of $g$, the convexity of $h^*$, and \eqref{eq:smooth-bregman} to obtain:
\begin{align*}
&\quad\calL(x^k,s)-\calL(x,s^{k+1}) \\
&=f(x^k)-f(x)+g(x^k)-g(x)+\ip{Ax^k}{s}-\ip{Ax}{s^{k+1}}
    +h^*(s^{k+1})-h^*(s)\\
&\le
\ip{\nabla f(x^k)}{x^k-x}
-\frac1{2L}\norm{\nabla f(x^k)-\nabla f(x)}^2
+\ip{q_g^k}{x^k-x}\\
&\quad
+\ip{q_h^{k+1}}{s^{k+1}-s}
+\ip{Ax^k}{s}-\ip{Ax}{s^{k+1}}\\
&\le
\ip{\nabla f(x^k) + q_g^k + A^\top s^{k+1}}{x^k-x}
-\frac1{2L}\norm{\nabla f(x^k)-\nabla f(x)}^2\\
&\quad
    +\ip{q_h^{k+1}-Ax^k}{s^{k+1}-s}.
    \end{align*}
Therefore, using \eqref{eq:pd3o-id1} and \eqref{eq:pd3o-id2},
\begin{align}\label{eq:pd3o-gap-start}
&\calL(x^k,s)-\calL(x,s^{k+1}) \\
&\le
\frac1\tau\ip{\delta_z^k}{x^k-x-\tau A^\top(s^{k+1}-s)}
+\frac1\eta\ip{H\delta_s^k}{s^{k+1}-s}
-\frac1{2L}\norm{\nabla f(x^k)-\nabla f(x)}^2.
\notag
\end{align}
By the definitions of $z^{k+1}$ and $z(x,s)$,
\begin{equation}\label{eq:pd3o-transform}
    x^k-x-\tau A^\top(s^{k+1}-s)
    =z^{k+1}-z(x,s)+\tau(\nabla f(x^k)-\nabla f(x)).
\end{equation}
Substituting \eqref{eq:pd3o-transform} into \eqref{eq:pd3o-gap-start} yields
\begin{align*}
&\calL(x^k,s)-\calL(x,s^{k+1})\\
&\le
\frac1\tau\ip{\delta_z^k}{z^{k+1}-z(x,s)}
+\frac1\eta\ip{H\delta_s^k}{s^{k+1}-s}
+\ip{\delta_z^k}{\nabla f(x^k)-\nabla f(x)}
-\frac1{2L}\norm{\nabla f(x^k)-\nabla f(x)}^2.
\end{align*}
Using the identity $\ip{a-b}{b-c}=\frac12\|a-c\|^2-\frac12\|b-c\|^2-\frac12\|a-b\|^2$ twice gives
\begin{align*}
\frac1\tau\ip{\delta_z^k}{z^{k+1}-z(x,s)}
&=\frac1{2\tau}\norm{z^k-z(x,s)}^2
 -\frac1{2\tau}\norm{z^{k+1}-z(x,s)}^2
 -\frac1{2\tau}\norm{\delta_z^k}^2,\\
\frac1\eta\ip{H\delta_s^k}{s^{k+1}-s}
&=\frac1{2\eta}\normM{s^k-s}{H}^2
 -\frac1{2\eta}\normM{s^{k+1}-s}{H}^2
 -\frac1{2\eta}\normM{\delta_s^k}{H}^2.
\end{align*}
Finally,
\[
    \ip{\delta_z^k}{\nabla f(x^k)-\nabla f(x)}
    -\frac1{2L}\norm{\nabla f(x^k)-\nabla f(x)}^2
    \le \frac L2\norm{\delta_z^k}^2.
\]
Consequently,
\begin{equation*}
\begin{aligned}
&\calL(x^k,s)-\calL(x,s^{k+1})\\
&\le
D_k^{\PD}(x,s)-D_{k+1}^{\PD}(x,s)
+\frac{2\alpha-1}{2\tau}\norm{\delta_z^k}^2
-\frac1{2\eta}\normM{\delta_s^k}{H}^2.
\end{aligned}
\end{equation*}
Thus the residual can be chosen as
\begin{equation}\label{eq:R-pd3o}
    R_k^{\PD}:=\frac{(2\alpha-1)_+}{2\tau}\norm{\delta_z^k}^2.
\end{equation}

\subsubsection{Lyapunov descent and residual domination}

Let
\[
    q_g^\star:=-\nabla f(x^\star)-A^\top s^\star\in\partial g(x^\star),
    \qquad
    q_h^\star:=Ax^\star\in\partial h^*(s^\star).
\]
Monotonicity of $\partial g$ and $\partial h^*$ gives
\begin{equation}\label{eq:pd3o-monotonicity}
    \ip{x^k-x^\star}{q_g^k-q_g^\star}
    +\ip{s^{k+1}-s^\star}{q_h^{k+1}-q_h^\star}\ge0.
\end{equation}
Using \eqref{eq:pd3o-id1}--\eqref{eq:pd3o-id2}, \eqref{eq:pd3o-monotonicity} becomes
\begin{align*}
0\le{}&
\frac1\tau\ip{\delta_z^k}{x^k-x^\star-\tau A^\top(s^{k+1}-s^\star)}
+\frac1\eta\ip{H\delta_s^k}{s^{k+1}-s^\star}\\
&-\ip{x^k-x^\star}{\nabla f(x^k)-\nabla f(x^\star)}.
\end{align*}
Since
\[
    x^k-x^\star-\tau A^\top(s^{k+1}-s^\star)
    =z^{k+1}-z^\star+\tau(\nabla f(x^k)-\nabla f(x^\star)),
\]
we obtain
\begin{align*}
0\le{}&
V_k^{\PD}-V_{k+1}^{\PD}
-\frac1{2\tau}\norm{\delta_z^k}^2
-\frac1{2\eta}\normM{\delta_s^k}{H}^2 \\
&+\ip{\delta_z^k}{\nabla f(x^k)-\nabla f(x^\star)}
-\ip{x^k-x^\star}{\nabla f(x^k)-\nabla f(x^\star)}.
\notag
\end{align*}
By cocoercivity and Young's inequality,
\begin{align*}
&\ip{\delta_z^k}{\nabla f(x^k)-\nabla f(x^\star)}
-\ip{x^k-x^\star}{\nabla f(x^k)-\nabla f(x^\star)}\\
&\le
\ip{\delta_z^k}{\nabla f(x^k)-\nabla f(x^\star)}
-\frac1L\norm{\nabla f(x^k)-\nabla f(x^\star)}^2
\le \frac L4\norm{\delta_z^k}^2.
\end{align*}
Therefore
\begin{equation}\label{eq:descent-pd3o}
    V_{k+1}^{\PD}
    \le V_k^{\PD}
    -\frac{1-\tau L/2}{2\tau}\norm{\delta_z^k}^2
    -\frac1{2\eta}\normM{\delta_s^k}{H}^2.
\end{equation}
Since $\alpha=\tau L/2<1$, combining \eqref{eq:R-pd3o} and \eqref{eq:descent-pd3o},
\begin{equation*}
    R_k^{\PD}
    \le \rho_{\PD}(V_k^{\PD}-V_{k+1}^{\PD}),
    \qquad
    \boxed{\rho_{\PD}:=\frac{(2\alpha-1)_+}{1-\alpha}}.
\end{equation*}
Since $H\succeq0$, $D_k^{\PD}(x,s)\ge0$, Theorem~\ref{thm:abstract-transfer} gives
\begin{equation*}
    \calL(\bar x^K,s)-\calL(x,\bar s^K)
    \le
    \frac{D_0^{\PD}(x,s)+\rho_{\PD}V_0^{\PD}}{K},
\end{equation*}
where
\[
    \bar x^K:=\frac1K\sum_{k=0}^{K-1}x^{k},    \qquad     \bar s^K:=\frac1K\sum_{k=0}^{K-1}s^{k+1}.
\]
This proves the PD3O part of Theorem~\ref{thm:unified}.

\subsection{Verification for PDDY/AFBA}\label{sec:pddy}

There exists \(\theta\in(3/4,1]\) such that
\begin{equation}\label{eq:pddy-condition}
    \alpha<1,\qquad
    \theta\beta\le1,\qquad
    \left[\alpha+4(1-\alpha)\theta(1-\theta)\right]\beta<1.
\end{equation}
Choose \(\varepsilon>0\) small enough so that
\begin{equation}\label{eq:pddy-eps-condition}
     c_{\theta}
    :=
    \alpha
    +4(1-\alpha)\theta(1-\theta)
    +4\varepsilon(1-\alpha)(1-\theta)<1/\beta.
\end{equation}
Define
\[
    M_\theta:=\Id-(2\theta-1)\tau\eta AA^\top, \qquad
     N_{\theta}
    :=
    \Id- c_{\theta}\tau\eta AA^\top.
\]
Then 
\[
    M_\theta\succeq0,\qquad
     N_{\theta}
    \succeq
    \nu_{\theta}\Id,
    \qquad
    \nu_{\theta}
    :=
    1- c_{\theta}\beta>0.
\]
The distance term for the one-step gap is
\begin{equation*}
    D_k^{\PDDY}(x,s)
    :=\frac1{2\tau}\norm{x^k-x}^2+\frac1{2\eta}\norm{s^k-s}^2.
\end{equation*}
The convergence Lyapunov is
\begin{equation*}
    V_k^{\PDDY}
    :=\frac1{2\tau}\norm{x^k-x^\star}^2
      +\frac1{2\eta}\norm{s^k-s^\star}^2
      +(1-\alpha)\frac{2\theta-1}{2\eta}\normM{\delta_s^{k-1}}{M_\theta}^2,
\end{equation*}
with the convention $s^{-1}=s^0$ if needed.

For notational simplicity, the displayed inequalities are written from $k=0$.  If the initial variables are not chosen so that the subgradient $q_g^0\in\partial g(y^0)$ below is defined by the same formula, one may start the summation from $k=1$; this changes only the constant in the final $O(1/K)$ bound.

\subsubsection{One-step gap inequality}

The proximal steps imply
\begin{align*}
    q_g^k
    &:=\frac1\tau(x^k-y^k)-A^\top s^k-\nabla f(x^k)
      \in\partial g(y^k), \\
    q_h^{k+1}
    &:=Ay^k+\frac1\eta\delta_s^k
      \in\partial h^*(s^{k+1}).
\end{align*}
Using convexity of $g$, convexity of $h^*$, and smoothness of $f$ at $x^k$,
\begin{align*}
    g(y^k)-g(x)&\le \ip{q_g^k}{y^k-x},\\
    f(y^k)-f(x)&\le \ip{\nabla f(x^k)}{y^k-x}+\frac L2\norm{y^k-x^k}^2,\\
    h^*(s^{k+1})-h^*(s)&\le \ip{q_h^{k+1}}{s^{k+1}-s}.
\end{align*}
Combining these inequalities gives
\begin{align}\label{eq:pddy-gap-pre}
&\calL(y^k,s)-\calL(x,s^{k+1})\\
&\le
\frac L2\norm{y^k-x^k}^2
+\frac1\eta\ip{\delta_s^k}{s^{k+1}-s}
+\frac1\tau\ip{x^k-y^k}{y^k-x}
+\ip{x-y^k}{A^\top\delta_s^k}.
\notag
\end{align}
Using \eqref{eq:pddy-x}, the last two terms in \eqref{eq:pddy-gap-pre} can be reorganized as
\begin{align*}
&\frac1\tau\ip{x^k-y^k}{y^k-x}
+\ip{x-y^k}{A^\top \delta_s^k}\\
&=\frac1\tau\ip{\delta_x^k}{x^{k+1}-x}
  -\ip{\delta_x^k}{A^\top \delta_s^k}\\
  &=\frac1\tau\ip{\delta_x^k}{x^{k+1}-x}+\frac1{2\tau}\norm{\delta_x^k}^2-\frac1{2\tau}\norm{x^k-y^k}^2
   +\frac1{2\eta}\normM{\delta_s^k}{\tau\eta AA^\top}^2.
\end{align*}
Substituting into \eqref{eq:pddy-gap-pre}, we obtain
\begin{equation}\label{eq:gap-pddy}
\begin{aligned}
&\calL(y^k,s)-\calL(x,s^{k+1})\\
&\le D_k^{\PDDY}(x,s)-D_{k+1}^{\PDDY}(x,s)
+\frac{2\alpha-1}{2\tau}\norm{x^k-y^k}^2
-\frac1{2\eta}\normM{\delta_s^k}{\Id-\tau\eta AA^\top}^2.
\end{aligned}
\end{equation}
This is the standard one-step gap inequality. The difference from the standard proof is that the last two terms are allowed to be positive and are later controlled by the Lyapunov descent.

From \eqref{eq:pddy-x},
\[
    x^k-y^k=\delta_x^k+\tau A^\top\delta_s^k,
\]
so
\begin{equation}\label{eq:pddy-bound-x-y}
    \norm{x^k-y^k}^2
    \le 2\norm{\delta_x^k}^2
      +2\tau^2\|A\|^2\norm{\delta_s^k}^2.
\end{equation}
Also,
\begin{equation}\label{eq:pddy-bound-indef}
    -\normM{\delta_s^k}{\Id-\tau\eta AA^\top}^2
    \le (\beta-1)_+\norm{\delta_s^k}^2.
\end{equation}
Combining \eqref{eq:gap-pddy}, \eqref{eq:pddy-bound-x-y}, and \eqref{eq:pddy-bound-indef}, we get
\begin{equation*}
    \calL(y^k,s)-\calL(x,s^{k+1})
    \le D_k^{\PDDY}(x,s)-D_{k+1}^{\PDDY}(x,s)+R_k^{\PDDY},
\end{equation*}
where
\begin{equation*}
\begin{aligned}
R_k^{\PDDY}:={}&2p_\tau\norm{\delta_x^k}^2+\left(2p_\tau\tau^2\|A\|^2+\frac{(\beta-1)_+}{2\eta}\right)
\norm{\delta_s^k}^2, \qquad
    p_\tau:=\frac{(2\alpha-1)_+}{2\tau}.
\end{aligned}
\end{equation*}

\subsubsection{Convergence Lyapunov descent}

Define
\begin{equation*}
    T_1:=-\ip{A^\top\delta_s^k}{\delta_x^k}
    +\ip{x^\star-y^k}{\nabla f(x^k)-\nabla f(x^\star)}.
\end{equation*}
The PDDY fundamental equality is
\begin{align}\label{eq:pddy-fundamental}
&\ip{y^k-x^\star}{q_g^k-q_g^\star}
 +\ip{s^{k+1}-s^\star}{q_h^{k+1}-q_h^\star}\\
&=\frac1\tau\ip{x^{k+1}-x^\star}{\delta_x^k}
 +\frac1\eta\ip{s^{k+1}-s^\star}{\delta_s^k}
 +T_1.
\notag
\end{align}
Here $q_g^\star=-A^\top s^\star-\nabla f(x^\star)$ and $q_h^\star=Ax^\star$.  By monotonicity, the left-hand side is nonnegative.  
Following the proof of~\cite[Lemma 3]{YanLi2024}, we first have
\begin{align}
T_1
&\le
\frac{\alpha}{2\tau}\norm{\delta_x^k}^2
+\frac{\alpha\tau}{2}\norm{A^\top\delta_s^k}^2
-(1-\alpha)\ip{A\delta_x^k}{\delta_s^k}.
\label{eq:pddy-T1-first}
\end{align}
Then, for the cross term, we have the upper bound~\cite[Equation (21)]{YanLi2024}
\begin{equation}\label{eq:pddy-cross-history}
    -\ip{A\delta_x^k}{\delta_s^k}
    \le
    \frac1{2\eta}
    \left(
        \normM{\delta_s^{k-1}}{M_\theta}^2
        -
        \normM{\delta_s^k}{M_\theta}^2
    \right)
    +2(1-\theta)\tau\norm{A^\top\delta_s^k}^2,
\end{equation}
and Young's inequality gives
\begin{equation}\label{eq:pddy-cross-young}
    -\ip{A\delta_x^k}{\delta_s^k}
    \le
    (1-\theta+\varepsilon)\tau\norm{A^\top\delta_s^k}^2+\frac1{4(1-\theta+\varepsilon)\tau}\norm{\delta_x^k}^2 .
\end{equation}
The difference here is that we use a weighted split of the cross term. Taking the weighted combination of the two upper bounds gives
\begin{align}
-\ip{A\delta_x^k}{\delta_s^k}
\le{}&
(2\theta-1)
\frac1{2\eta}
\left(
        \normM{\delta_s^{k-1}}{M_\theta}^2
        -
        \normM{\delta_s^k}{M_\theta}^2
\right)+
2(2\theta-1)(1-\theta)\tau\norm{A^\top\delta_s^k}^2
\notag\\
&+
2(1-\theta)(1-\theta+\varepsilon)\tau\norm{A^\top\delta_s^k}^2+
\frac{(1-\theta)}
     {2(1-\theta+\varepsilon)\tau}
     \norm{\delta_x^k}^2 .
\label{eq:pddy-cross-split}
\end{align}
Thus, we have 
\begin{align}
T_1
\le{}&
\frac1{2\tau}\norm{\delta_x^k}^2
-
\frac{(1-\alpha)\varepsilon}
     {2\tau(1-\theta+\varepsilon)}
     \norm{\delta_x^k}^2
\notag\\
&+
(1-\alpha)(2\theta-1)
\frac1{2\eta}
\left(
        \normM{\delta_s^{k-1}}{M_\theta}^2
        -
        \normM{\delta_s^k}{M_\theta}^2
\right)+
\frac{\tau}{2}c_\theta\norm{A^\top\delta_s^k}^2.
\label{eq:pddy-T1-bound}
\end{align}
Combining \eqref{eq:pddy-fundamental} and \eqref{eq:pddy-T1-bound}, using the distance identity for the first two inner products and $N_\theta=\Id-c_\theta\tau\eta AA^\top$, yields
\begin{align*}
V_{k+1}^{\PDDY}
\le{}& V_k^{\PDDY}
-
\frac{(1-\alpha)\varepsilon}
     {2\tau(1-\theta+\varepsilon)}
     \norm{\delta_x^k}^2
-\frac1{2\eta}\normM{\delta_s^k}{N_\theta}^2\\
\leq{}&V_k^{\PDDY}
    -a_\theta^{\PDDY}\norm{\delta_x^k}^2
    -b_\theta^{\PDDY}\norm{\delta_s^k}^2,
\end{align*}
where
\begin{equation}\label{eq:pddy-new-a-b}
    a_{\theta}^{\PDDY}
    :=
    \frac{(1-\alpha)\varepsilon}
         {2\tau(1-\theta+\varepsilon)}
    >0,
    \qquad
    b_{\theta}^{\PDDY}
    :=
    \frac{\nu_{\theta}}{2\eta}
    >0.
\end{equation}
Therefore the residual from the one-step gap inequality satisfies
\begin{equation*}
    R_k^{\PDDY}
    \le \rho_{\PDDY}(V_k^{\PDDY}-V_{k+1}^{\PDDY}),
\end{equation*}
where
\begin{equation*}
    \boxed{\rho_{\PDDY}
    :=\max\left\{
       \frac{2p_\tau}{a_\theta^{\PDDY}},
       \frac{2p_\tau\tau^2\|A\|^2+(\beta-1)_+/(2\eta)}{b_\theta^{\PDDY}}
    \right\}}.
\end{equation*}
Since $D_k^{\PDDY}(x,s)\ge0$, Theorem~\ref{thm:abstract-transfer} gives
\begin{equation*}
    \calL(\bar y^K,s)-\calL(x,\bar s^K)
    \le
    \frac{D_0^{\PDDY}(x,s)+\rho_{\PDDY}V_0^{\PDDY}}{K},
\end{equation*}
where
\[
    \bar y^K:=\frac1K\sum_{k=0}^{K-1}y^{k},    \qquad     \bar s^K:=\frac1K\sum_{k=0}^{K-1}s^{k+1}.
\]
This proves the PDDY/AFBA part of Theorem~\ref{thm:unified}.

\subsection{Verification for PAPC/PDFP$^2$O}\label{sec:papc}

Assume $g\equiv0$.  We can find $\theta\in(3/4,1]$ such that
\begin{equation*}
    \alpha<1,
    \qquad
    \theta\beta< 1.
\end{equation*}
Define
\begin{equation*}
    R_0:=\Id-\tau\eta AA^\top,
    \qquad
    R_\theta:=\Id-(2\theta-1)\tau\eta AA^\top,
    \qquad
    S_\theta:=\Id-\theta\tau\eta AA^\top.
\end{equation*}
Then $R_\theta\succ0$ and $S_\theta\succ(1-\theta\beta)I$.

For arbitrary test points $(x,s)$, set
\begin{equation*}
    D_k^{\PAPC}(x,s)
    :=\frac1{2\tau}\norm{x^k-x}^2+\frac1{2\eta}\normM{s^k-s}{R_0}^2.
\end{equation*}
This distance can be indefinite when $\beta>1$.  The positive Lyapunov is
\begin{equation*}
    V_k^{\PAPC}
    :=\frac1{2\tau}\norm{x^k-x^\star}^2+\frac1{2\eta}\normM{s^k-s^\star}{R_\theta}^2.
\end{equation*}

\subsubsection{One-step gap inequality}

The PAPC updates imply the special subgradient identity
\begin{equation}\label{eq:papc-qh}
    q_h^{k+1}:=Ax^{k+1}+\frac1\eta R_0\delta_s^k
    \in\partial h^*(s^{k+1}).
\end{equation}
Using smoothness of $f$ and convexity of $h^*$,
\begin{align*}
&\quad\calL(x^{k+1},s)-\calL(x,s^{k+1})\\
&=f(x^{k+1})-f(x)+\ip{Ax^{k+1}}{s}-\ip{Ax}{s^{k+1}}
+h^*(s^{k+1})-h^*(s)\\
&\le
\ip{\nabla f(x^k)}{x^{k+1}-x}+\frac L2\norm{\delta_x^k}^2
+\ip{q_h^{k+1}}{s^{k+1}-s}\\
&\quad
+\ip{Ax^{k+1}}{s}-\ip{Ax}{s^{k+1}}.
\end{align*}
Substituting \eqref{eq:papc-qh}, the inner product terms combine to
\[
    \ip{\nabla f(x^k)+A^\top s^{k+1}}{x^{k+1}-x}
    +\frac1\eta\ip{R_0\delta_s^k}{s^{k+1}-s}.
\]
By the primal update,
\[
    \nabla f(x^k)+A^\top s^{k+1}=\frac1\tau\delta_x^k.
\]
Thus
\begin{align*}
&\calL(x^{k+1},s)-\calL(x,s^{k+1})\\
&\le
\frac1\tau\ip{\delta_x^k}{x^{k+1}-x}
+\frac1\eta\ip{R_0\delta_s^k}{s^{k+1}-s}
+\frac L2\norm{\delta_x^k}^2.
\end{align*}
Applying the distance identity in the primal and dual variables gives
\begin{equation*}
\begin{aligned}
\calL(x^{k+1},s)-\calL(x,s^{k+1})
&\le
D_k^{\PAPC}(x,s)-D_{k+1}^{\PAPC}(x,s)
+\frac{2\alpha-1}{2\tau}\norm{\delta_x^k}^2
-\frac1{2\eta}\normM{\delta_s^k}{R_0}^2.
\end{aligned}
\end{equation*}
Let
\begin{equation*}
    p_\tau:=\frac{(2\alpha-1)_+}{2\tau},
    \qquad
    p_s:=\frac{(\beta-1)_+}{2\eta}.
\end{equation*}
Since
\[
    -\frac1{2\eta}\normM{\delta_s^k}{R_0}^2
    \le p_s\norm{\delta_s^k}^2,
\]
we may take
\begin{equation}\label{eq:R-papc}
    R_k^{\PAPC}
    :=p_\tau\norm{\delta_x^k}^2+p_s\norm{\delta_s^k}^2.
\end{equation}

\subsubsection{Large-step PAPC descent}

In the notation of the large-step PAPC analysis in~\cite[Theorem 1]{LiYan2021}, with $P=D=\Id$ and with no infimal-convolution term, we have
\begin{align}
    V_{k+1}^{\PAPC} -V_k^{\PAPC}
    &\le 
    -a_\theta^{\PAPC}\norm{\delta_x^k}^2
    -\frac{1-\alpha}{2\eta}\norm{\delta_s^k}_{S_\theta}^2,\nonumber\\
    &\leq 
    -a_\theta^{\PAPC}\norm{\delta_x^k}^2
    -b_\theta^{\PAPC}\norm{\delta_s^k}^2,\label{eq:descent-papc}
\end{align}
where
\begin{equation*}
    a_\theta^{\PAPC}
    :=\frac1{2\tau}\,
      \frac{(4\theta-3)(1-\alpha)}{1-\alpha+(4\theta-3)\alpha}>0,
    \qquad
    b_\theta^{\PAPC}
    :=\frac{1-\alpha}{2\eta}(1-\theta\beta)>0.
\end{equation*}
The positivity follows from $\theta>3/4$, $\alpha <1$, and $\theta\beta<1$.

Combining \eqref{eq:R-papc} with \eqref{eq:descent-papc},
\begin{equation*}
    R_k^{\PAPC}
    \le \rho_{\PAPC}(V_k^{\PAPC}-V_{k+1}^{\PAPC}),
\end{equation*}
where
\begin{equation*}
    \boxed{\rho_{\PAPC}
    :=\max\left\{
    \frac{p_\tau}{a_\theta^{\PAPC}},
    \frac{p_s}{b_\theta^{\PAPC}}
    \right\}}.
\end{equation*}

\subsubsection{Lower bound for the indefinite distance}

When $\beta\le1$, $R_0\succeq0$, and $D_k^{\PAPC}(x,s)\ge0$.  The interesting large-step case is $\beta>1$, where $R_0$ may be indefinite.  Since
\begin{equation*}
    R_0\succeq -\kappa_0 R_\theta,
    \qquad
    \kappa_0:=\frac{(\beta-1)_+}{1-\theta\beta+(1-\theta)\beta}\geq 0,
\end{equation*}
we have
\begin{align*}
D_k^{\PAPC}(x,s)
&\ge -\frac{\kappa_0}{2\eta}\normM{s^k-s}{R_\theta}^2\\
&\ge -\frac{\kappa_0}{\eta}\normM{s^k-s^\star}{R_\theta}^2
      -\frac{\kappa_0}{\eta}\normM{s-s^\star}{R_\theta}^2\\
&\ge -2\kappa_0 V_k^{\PAPC}
      -\frac{\kappa_0}{\eta}\normM{s-s^\star}{R_\theta}^2.
\end{align*}
Thus \eqref{eq:H3} holds with
\begin{equation*}
    \kappa_{\PAPC}:=2\kappa_0,
    \qquad
    B_{\PAPC}(x,s):=\frac{\kappa_0}{\eta}\normM{s-s^\star}{R_\theta}^2.
\end{equation*}
Applying Theorem~\ref{thm:abstract-transfer} gives
\begin{equation*}
\begin{aligned}
&\calL\left(\bar x^K,s\right)
 -\calL\left(x,\bar s^K\right)\le
\frac{D_0^{\PAPC}(x,s)+B_{\PAPC}(x,s)
 +(\rho_{\PAPC}+\kappa_{\PAPC})V_0^{\PAPC}}{K},
\end{aligned}
\end{equation*}
where
\[
    \bar x^K:=\frac1K\sum_{k=0}^{K-1}x^{k+1},    \qquad     \bar s^K:=\frac1K\sum_{k=0}^{K-1}s^{k+1}.
\]
This proves the PAPC/PDFP$^2$O part of Theorem~\ref{thm:unified}.

\section{A counterexample to the separated large-step condition}
\label{sec:counterexample}

In this section, we provide an example showing that the separated condition
\[
    \tau L<2,
    \qquad
    \tau\eta\|A\|^2<\frac43
\]
cannot guarantee convergence for the general three-function problem.
Define two subspaces
\[
    \mathcal P:=\{x:x_1=x_2/3\},\qquad \mathcal Q:=\{x:x_1=x_2\}.
\]
Let
\[
    A=I_2,
    \qquad
    f(x)=\frac12 x_1^2,\qquad  g(x):=\iota_{\mathcal Q}(x),\qquad h^*(s):=\iota_{\mathcal P}(s),
\]
where \(\iota_C\) denotes the indicator function of the set \(C\).  Since \(\mathcal P\) is a subspace, its conjugate satisfies
\[
    h=(h^*)^*=\iota_{\mathcal P^\perp}.
\]
Therefore, the primal problem is
\[
    \min_{x\in\mathbb R^2}
    \frac12 x_1^2+\iota_{\mathcal Q}(x)+\iota_{\mathcal P^\perp}(x),
\]
whose unique primal solution is
\[
    x^*=0.
\]
Moreover, \((x^*,s^*)=(0,0)\) is a saddle point.  The gradient of \(f\) is Lipschitz continuous with \(L=1,\)
and \(\|A\|^2=1\).

Choose
\[
    \tau=\frac{39}{20},
    \qquad
    \eta=\frac23.
\]
Then
\[
    \alpha=\frac{\tau L}{2}=\frac{39}{40}<1,
    \qquad
    \beta=\tau\eta\|A\|^2=\frac{13}{10}<\frac43.
\]
Thus the example satisfies the separated large-step condition.

\subsection{Divergence of PDDY/AFBA}

For this example, the PDDY/AFBA iteration becomes
\[
\begin{aligned}
    s^{k+1}&=P_{\mathcal P}(s^k+\eta y^k),\\
    x^{k+1}&=y^k-\tau(s^{k+1}-s^k),\\
    y^{k+1}&=P_{\mathcal Q}\bigl(x^{k+1}-\tau s^{k+1}-\tau \nabla f(x^{k+1})\bigr),
\end{aligned}
\]
where \(P_{\mathcal P}\) and \(P_{\mathcal Q}\) are the orthogonal projections onto \(\mathcal P\) and \(\mathcal Q\), respectively.

After one iteration, we have
\[
    s^k\in\mathcal P,
    \qquad
    y^k\in\mathcal Q.
\]
Hence, we may write
\[
    s^k=\xi_k \begin{pmatrix}1\\3\end{pmatrix},
    \qquad
    y^k=\zeta_k \begin{pmatrix}1\\1\end{pmatrix}.
\]
A direct calculation gives
\[\begin{aligned}
    \xi_{k+1}&=\xi_k+\frac{2}{5}\eta \zeta_k=\xi_k+\frac{4}{15} \zeta_k,\\
    \zeta_{k+1}&=-2\tau \xi_k+
    \Bigl(1-\frac{1}{2}\tau -\frac{8}{5}\tau\eta +\frac{1}{5}\tau^2\eta \Bigr)\zeta_k=-\frac{39}{10} \xi_k-\frac{387}{250}\zeta_k.
\end{aligned}
\]   
Therefore
\[
    \binom{\xi_{k+1}}{\zeta_{k+1}}
    =M_{\rm PDDY}
    \binom{\xi_k}{\zeta_k}, \qquad\mbox{where } M_{\rm PDDY}
    =
    \begin{pmatrix}
        1 & {4}/{15}\\[2mm]
        -3.9 &  -1.548
    \end{pmatrix}.
\]
The matrix \(M_{\rm PDDY}\) has a real eigenvalue smaller than \(-1\). Consequently, the PDDY/AFBA iteration diverges for any initialization with a nonzero component in the corresponding unstable eigendirection.

\subsection{Divergence of standard PD3O}

The same example also destabilizes the standard PD3O iteration
\[
\begin{aligned}
    x^k&=P_{\mathcal Q}z^k,\\
    s^{k+1}
    &=P_{\mathcal P}\Bigl((I-\tau\eta I)s^k
    +\eta\bigl(2x^k-z^k-\tau\nabla f(x^k)\bigr)\Bigr),\\
    z^{k+1}&=x^k-\tau\nabla f(x^k)-\tau s^{k+1}.
\end{aligned}
\]
Since \(s^k\in\mathcal P\) after one iteration, write
\[
    s^k=\xi_k \begin{pmatrix}1\\ 3\end{pmatrix},
    \qquad
    z^k=\binom{z_1^k}{z_2^k},\qquad w^k:=\begin{pmatrix}\xi_k\\ z_1^k\\ z_2^k\end{pmatrix}.
\]
Then we have 
\[
    w^{k+1}=M_{\rm PD3O}w^k,\qquad\text{where }M_{\rm PD3O}
    =
    \begin{pmatrix}
        -0.3 & 0.135 & {1}/{600}\\
        0.585 & -0.73825 & -0.47825\\
        1.755 & -0.28975 & 0.49025
    \end{pmatrix}.
\]
A direct computation shows that \(M_{\rm PD3O}\) has a real eigenvalue smaller than \(-1\). Thus, the standard PD3O iteration also diverges for suitable initializations.

\section{Conclusion}
\label{sec:conclusion}

We have shown that the usual halving of the primal smooth stepsize in ergodic primal-dual gap proofs is not intrinsic.  The residual-to-gap transfer theorem separates the one-step gap calculation from the Lyapunov descent: residual terms that are positive when $\tau >1/L$ can still be summed because they are controlled by the decrease of a Lyapunov function.  This yields gap rates with $\tau <2/L$ for Condat--V\~u, PD3O, AFBA/PDDY, and PAPC/PDFP$^2$O, while preserving their distinct product restrictions.

The counterexample in Section~\ref{sec:counterexample} shows that the PAPC large-step rectangle cannot be extended unchanged to the general three-function setting.  An interesting open direction is whether standard PD3O admits a relaxed product condition beyond $\tau\eta\|A\|^2\le1$ under a lifted Lyapunov involving both primal and dual history terms.  Such a result would require a different argument from the proof used here.

\section*{Acknowledgments}
This work was partially supported by the National Natural Science Foundation of China (72495131, 82441027), Guangdong Provincial Key Laboratory of Mathematical Foundations for Artificial Intelligence (2023B1212010001), Shenzhen Stability Science Program, and the Shenzhen Science and Technology Program under grant number JCYJ20250604141043020. The authors used ChatGPT Pro by OpenAI as an auxiliary tool to help explore the counterexample and organize parts of the proof. The authors independently verified all calculations and arguments and assume responsibility for all content.

\bibliographystyle{plain}
\bibliography{PDgap}

\end{document}